\documentclass[12pt,a4paper]{amsart}
\usepackage[utf8]{inputenc}

\usepackage[colorlinks=false,pdfborder={0 0 0},bookmarks=false,unicode=true]{hyperref}
\hypersetup{
  pdfauthor={Martin At. Stanev},
  pdftitle={Log-Convexity of Weighted Area Integral Means of Hp
    Functions on the Upper Half-plane},
  pdfsubject={},
  pdfkeywords={log-convex function, Hardy space on the upper half plane, weighted area integral means}
}

\title[Log-Convexity of WAIM of $H^p$ functions]{Log-Convexity of Weighted Area Integral Means of $H^p$ 
    Functions on the Upper Half-plane}
\author{Martin~At.~Stanev}
\email{martin\_stanev@yahoo.com}
\address{Department of Mathematics and Physics, University of Forestry, Sofia, Bulgaria}
\keywords{log-convexity,weighted area integral means,holomorphic function,upper half-plane}
\subjclass[2010]{30H10, 30H20}

\date{2017-08-05} 

\theoremstyle{plain}
\newtheorem{theorem}{Theorem}
\newtheorem{lemma}[theorem]{Lemma}

\theoremstyle{definition}
\newtheorem{definition}[theorem]{Definition}
\newtheorem{example}[theorem]{Example}

\theoremstyle{remark}
\newtheorem{remark}[theorem]{Remark}

\newcommand{\smooth}[2]{{\mathcal D}^{#1}(#2)}

\begin{document}

\begin{abstract} In the present work weighted area integral means \\ $M_{p,\varphi}(f;{\mathrm {Im}}z)$
are studied and it is proved that the function \\ $y\to \log M_{p,\varphi}(f;y)$ is convex in the case when $f$ belongs to a Hardy space on the upper half-plane.\end{abstract}

\maketitle

\section{Introduction}
In the present paper we study three weighted area integral means of holomorphic on the upper half plane functions. They are defined as follows 
\begin{gather*}
M^{(1)}_{p,\varphi}(f;y)=\frac{\int_1^y \varphi'(t) \,\,\int_{-\infty}^{+\infty} |f(x+i t)|^p dx\,\,dt}{\int_1^y \varphi'(t) dt },\\
M^{(0)}_{p,\varphi}(f;y)=\frac{\int_0^y \varphi'(t) \,\,\int_{-\infty}^{+\infty} |f(x+i t)|^p dx\,\,dt}{\int_0^y \varphi'(t) dt },\\
M^{(\infty)}_{p,\varphi}(f;y)=\frac{\int_y^{+\infty} \varphi'(t) \,\,\int_{-\infty}^{+\infty} |f(x+i t)|^p dx\,\,dt}{\int_y^{+\infty} \varphi'(t) dt }
\end{gather*}
where $p>0$, $y>0$ and the functions $f$ and $\varphi$ are such that the integrals exist and the fraction can be defined as a continuous function on $(0;+\infty)$.

The goal is to find specific conditions on the functions $f$ and $\varphi$ under which each one of these three weighted area integral means is log-convex on $(0;+\infty)$.
This goal is partially achieved in theorems~\ref{th:twelve},~\ref{th:thirteen},~\ref{th:fourteen},~\ref{th:fifteen}  where some sufficient conditions are presented. Our theorems show
that in the case when $f$ belongs to the Hardy space $H^p$, $2\le p<+\infty$, these three weighted area integral means are similar to the classical integral
means
\[
M_p^p(f;y)=\int_{-\infty}^{+\infty} |f(x+i y)|^p dx, \quad y\in(0;+\infty)
\]
in terms of its monotonic growth and convexity behavior. 
Moreover, there is a specific weight $\varphi$ and a specific holomorphic function $f$ such that $f$ does not belong to any Hardy space and nevertheless such a similarity still exists.

Some of results in this paper are presented on the Second International Conference
``Mathematics Days in Sofia'', July 10-14, 2017, Sofia, Bulgaria. 

During the period 2011--2016, there is a series of papers by Ch.~Wang, J.~Xiao and K.~Zhu on weighted area integral means. In \cite{XiaoZhu2011} volume integral means
of holomorphic in the unit ball of ${\mathbb C}^n$ functions are studied. Among various results they stated a conjecture about convexity of $\log M_{p,\alpha}(f,r)$ in $\log r$. In \cite{XiaoXu2011} authors study monotonic growth and logarithmic convexity of integral means which are important from a geometric point of view. In~\cite{XiaoZhu2011}, \cite{WangZhu1101}, \cite{WangXiaoZhu1308} authors prove theorems
about convexity of $\log$ of a weighted area integral mean in $\log r$ in the case of holomorphic functions in the unit disk of ${\mathbb {C}}$. They consider the weight function $\varphi$ with $\varphi'(|z|^2)=(1-|z|^2)^{\alpha}$. In \cite{WangXiao1301}, \cite{WangXiao1405},
\cite{WangXiao1600} authors study the case when $f$ is an entire function on ${\mathbb C}$ and the weight function $\varphi$ with $\varphi'(|z|^2)=e^{-\alpha |z|^2}$. 

Note that the case of holomorphic functions on the upper half plane remains unexplored. Thus, the present paper contains theorems about weighted area integral means in a new case. We apply the method demonstrated in~\cite{WangXiaoZhu1308} and modify it with some details that are relevant to our goals.

A great deal of our computations are done and checked with a freeware open-source computer algebra system Maxima (wxMaxima) which is published at \href{http://maxima.sf.net}{http://maxima.sf.net}.

\section{Definitions}
\begin{definition}\label{defABC}
Let $I$, $I\subset (-\infty; +\infty)$, be a non-empty open interval, and $\smooth{n}{I}$ stand for the class of all real valued functions such that have a finite $n\mbox{--th}$ derivative everywhere in $I$. If the functions $q:I\to (0;+\infty)$, $\varphi:I\to (-\infty; +\infty)$ and $M:I\to (0;+\infty)$ are such that $q\in \smooth{2}{I}$, $\varphi \in \smooth{3}{I}$, $M\in \smooth{2}{I}$ then the functions $A$, $B_0$, $C_0$, $B$, $C$, $E_1$, $E_2$, $F_1$ and $F_2$ are defined as follows
\begin{gather*}
A=(q \varphi')' \varphi - q \varphi'^2,\quad B_0=(q \varphi')' \varphi^2, \quad C_0=q \varphi^2 \varphi'^2,\\
B=(q \varphi' M)' \varphi^2, \quad C=q \varphi^2 \varphi'^2 M^2,
\quad E_1=A^2 \varphi + E_2,\\
E_2=A q \varphi \varphi'^2 - B_0' q \varphi \varphi' + (q \varphi')' \varphi q (\varphi^2 \varphi')',\\
F_1=\frac{\,B-\sqrt{B^2-4AC} \,}{2A},\quad F_2=\frac{\,B+\sqrt{B^2-4AC} \,}{2A}
\end{gather*} 
where $'$ denotes differentiation and $F_{1,2}$ are defined on the subset of $I$ defined by the conditions $A\neq 0$, $B^2-4AC\ge 0$.
\end{definition}

Note that if $A\neq 0$, $B^2-4AC\ge 0$ then the functions $F_{1,2}$ are well defined real valued functions such that $A F_{i}^2-B F_{i} +C=0$, $i=1,2$.

\begin{example}\label{exa:weights} The following examples are used in the main theorems
\begin{enumerate}
\item If $I=(0;+\infty)$, $q(x)=1$, $\varphi(x)=\int_1^x t^{-a} dt$, where $x\in I$ and the constant $a>0$ then 
\[
A(x)=-x^{-a-1}(\varphi(x)+1),\quad E_1(x)=ax^{-2a-2} \varphi^2(x)
\]
\item If $I=(0;+\infty)$, $q(x)=1$, $\varphi(x)=\int_1^x e^{-t} dt$, where $x\in I$ then 
\[
A(x)=-e^{-x-1},\quad E_1(x)=e^{-2x-1} \varphi^2(x)
\] 
\item If $I=(0;+\infty)$, $q(x)=1$, $\varphi(x)=\int_0^x t^{-a} dt$, where $x\in I$ and the constant $a<1$ then 
\[
A(x)=(a-1)^{-1}x^{-2a},\quad E_1(x)=0
\] 
\item If $I=(0;+\infty)$, $q(x)=1$, $\varphi(x)=\int_0^x e^{-t} dt$, where $x\in I$ then 
\[
A(x)=-e^{-x},\quad E_1(x)=e^{-2x} \varphi^2(x)
\]  
\item If $I=(0;+\infty)$, $q(x)=1$, $\varphi(x)=-\int_x^{+\infty} t^{-a} dt$, where $x\in I$ and the constant $a>1$ then 
\[
A(x)=(a-1)^{-1}x^{-2a},\quad E_1(x)=0
\] 
\item If $I=(0;+\infty)$, $q(x)=1$, $\varphi(x)=-\int_x^{+\infty} e^{-t} dt$, where $x\in I$ then 
\[
A(x)=0,\quad E_1(x)=0
\] 
\item If $I=(0;+\infty)$, $q(x)=1$, $\varphi(x)=-\int_x^{+\infty} t^a e^{-t} dt$, where $x\in I$ and the constant $a<0$ then 
\[
A(x)>0,\quad E_1(x)>0
\] 
\end{enumerate}
The computations which are needed in (1)--(7) are simple and straight-forward and because of this they are omitted. 

Auxiliary example: $I=(0;+\infty)$, $q(x)=1$, $\varphi(x)=-\int_x^{+\infty} e^{t-e^t} dt$, where $x\in I$,
\[
A=-e^{x-2e^x}<0,\quad E_2=0,\quad E_1=A^2 \varphi+E_2<0
\]  
\end{example}

\section{Auxiliary results}
\begin{lemma}\label{E1E2lemma} Let $I$, $I\subset (-\infty; +\infty)$, be a non-empty open interval. If the functions $q:I\to (0;+\infty)$, $\varphi:I\to (-\infty; +\infty)$ are such that $q\in \smooth{2}{I}$, $\varphi \in \smooth{3}{I}$ and $\varphi'(x)\!\neq\! 0$ for all $x\!\in\! I$, then the following identities hold on $I$
\begin{align}
E_2&=A^2 \varphi - (A B_0 - A q (\varphi^2 \varphi')' + A' q \varphi^2 \varphi')\label{E2identity}\\
E_1&=-q^2 \varphi^2 \varphi'^3 \bigl( \frac{(q \varphi')' \varphi}{q \varphi'^2} \bigr)'\label{E1identity}
\end{align}
\end{lemma}

\begin{remark} Note that it follows by this lemma and the definition of $E_2$ that
\begin{equation}\label{rootF2}
\left| 
\begin{array}{r@{\hspace{3pt}}l}
A^2 \varphi +(-E_2)&=A B_0 - A q (\varphi^2 \varphi')' + A' q \varphi^2 \varphi'\\
\varphi (-E_2)&=(-1)(A C_0 - B_0' q \varphi^2 \varphi' + B_0 q (\varphi^2 \varphi')')
\end{array}
\right.
\end{equation}
\end{remark}

\begin{proof}[Proof of Lemma~\ref{E1E2lemma}]
Let the functions $q$ and $\varphi$ meet the conditions from the lemma. Identity~\eqref{E2identity} follows from the computation\footnote{$A' \varphi^2 - B_0' \varphi + C_0'=(-2A \varphi
+ B_0 )\varphi'$ follows from $A \varphi^2 - B_0 \varphi + C_0=0$ on $I$.}
\begin{multline*}
\bigl(A^2 \varphi - \bigl( AB_0- A q (\varphi^2 \varphi')' + A' q \varphi^2 \varphi' \bigr) -E_2 \bigr) \varphi \\
\shoveleft{=A^2 \varphi^2 - \bigl( AB_0- A q (\varphi^2 \varphi')'+ A' q \varphi^2 \varphi' \bigr)\varphi }\\
\shoveright{-\bigl( AC_0-B_0' q \varphi^2 \varphi'+ B_0 q (\varphi^2 \varphi')' \bigr)}\\
\shoveleft{=A(B_0 \varphi- C_0)  -  AB_0\varphi +\underline{ A q\varphi (\varphi^2 \varphi')'} -\underline{ \underline{  A' q \varphi^3 \varphi'}}}\\
\shoveright{- AC_0 +\underline{ \underline{  B_0' q \varphi^2 \varphi' }} - \underline {B_0 q (\varphi^2 \varphi')'}} \\
\shoveleft{=- 2AC_0+  q (\varphi^2 \varphi')' \bigl(A \varphi - B_0  \bigr) 
- q \varphi \varphi' \bigl(A' \varphi^2 - B_0' \varphi  \bigr)} \\
\shoveleft{= - \underline{2AC_0} - q (\varphi^2 \varphi')' q \varphi \varphi'^2 
 - q \varphi \varphi' \bigl( -\underline{2A \varphi \varphi'} + B_0 \varphi' - C_0' \bigr)}\\
= q \varphi \varphi' \bigl( -q \varphi' (\varphi^2 \varphi')' - (  B_0 \varphi' - C_0' )   \bigr)=0. 
\end{multline*}

In order to prove identity~\eqref{E1identity} note that 
by identity~\eqref{E2identity} it follows  that 
$E_1=2A^2 \varphi - AB_0 + A q(\varphi^2 \varphi')' -
A' q \varphi^2 \varphi'$. So,
\begin{multline*}
 E_1=A( 2A \varphi - B_0 + q(\varphi^2 \varphi')' ) - A' q \varphi^2 \varphi' \\
\shoveleft{ =A\bigl( 2(\underline{(q\varphi')'\varphi}-\underline{\underline{q\varphi'^2}}) \varphi- \underline{(q\varphi')'\varphi^2} + q(\underline{\underline{2\varphi \varphi'^2}}+ \varphi^2 \varphi'') \bigr)  - A' q \varphi^2 \varphi'}\\
\shoveleft{ =A\bigl( (q\varphi')'\varphi^2  + q \varphi^2 \varphi'' \bigr)  - A' q \varphi^2 \varphi'=
 A\bigl( (q\varphi')'\varphi'  + q \varphi' \varphi'' \bigr) \frac{\varphi^2}{\varphi'} - A' q \varphi^2 \varphi'}\\
\shoveleft{ =-\frac{\varphi^2}{\varphi'} (  A' q \varphi'^2 - A (q \varphi'^2)' ) = -\frac{\varphi^2}{\varphi'} (q \varphi'^2)^2 \Bigl( \frac{A}{q \varphi'^2} \Bigr)' = -q^2 \varphi^2 \varphi'^3 \Bigl( \frac{A}{q \varphi'^2} \Bigr)'}\\
=-q^2 \varphi^2 \varphi'^3 \Bigl( \frac{(q\varphi')'\varphi- q\varphi'^2}{q \varphi'^2} \Bigr)'=
  -q^2 \varphi^2 \varphi'^3 \Bigl( \frac{\varphi (q \varphi')'}{q \varphi'^2} \Bigr)'.
\end{multline*}
\end{proof}

\begin{lemma}\label{derivativesBC} Let $I$, $I\subset (-\infty; +\infty)$, be a non-empty open interval. Assume the functions $q:I\to (0;+\infty)$, $\varphi:I\to (-\infty; +\infty)$  and $M:I\to (0;+\infty)$ are such that $q\in \smooth{2}{I}$, $\varphi \in \smooth{3}{I}$, $M\in \smooth{2}{I}$.  Then the following identities hold on $I$
\begin{align}
B' q \varphi^2 \varphi' M&= B  (B- B_0 M) + B_0' q \varphi^2 \varphi' M^2 \hphantom{+ (\varphi^2 \varphi')' q M }  \label{derivativeB}\\
& \hphantom{+ (\varphi^2 \varphi')'(B) } + (\varphi^2 \varphi')' q M (B- B_0 M)  + C \varphi^2  \bigl(q \frac{M'}{M}\bigr)' \nonumber\\
C' q \varphi^2 \varphi' M &=  C q  ( \varphi^2 \varphi')' \, M + 2 B C - B_0 C M \label{derivativeC}
\end{align}
\end{lemma}

\begin{proof}[Proof of Lemma~\ref{derivativesBC}] The proof of  identity~\eqref{derivativeB} is as follows
\begin{multline*}
B' q \varphi^2 \varphi' M = \bigl( M (B_0+  q \varphi^2 \varphi'\, \frac{M'}{M}) \bigr)' q \varphi^2 \varphi'M=\\
\shoveleft{=M' (B_0+  q \varphi^2 \varphi' \frac{M'}{M}) q \varphi^2 \varphi' M  +
M (B_0'+   (\varphi^2 \varphi' \, q \frac{M'}{M})') q \varphi^2 \varphi' M}\\
\shoveleft{= M' B q \varphi^2 \varphi' + B_0' q \varphi^2 \varphi' M^2 }\\
\shoveright{ +(\varphi^2 \varphi')' \, q \frac{M'}{M} q \varphi^2 \varphi' M^2 + \varphi^2 \varphi' \, (q \frac{M'}{M})' q \varphi^2 \varphi' M^2 }\\
\shoveleft{= B  q \varphi^2 \varphi' M' + B_0' q \varphi^2 \varphi' M^2 }\\
\shoveright{  + (\varphi^2 \varphi')' \,q \varphi^2 \varphi' q M M' +  q \varphi^2 \varphi' \varphi^2 \varphi' M^2 \, (q \frac{M'}{M})' }\\
\shoveleft{= B  (B- B_0 M) + B_0' q \varphi^2 \varphi' M^2}\\
  + (\varphi^2 \varphi')' q M (B- B_0 M)  + C \varphi^2  (q \frac{M'}{M})'. 
\end{multline*}

The proof of identity~\eqref{derivativeC} is as follows
\begin{align*}
&C' q \varphi^2 \varphi' M =(q \varphi' \, \varphi' \varphi^2 \, M^2)' q \varphi^2 \varphi' M \\
&=\bigl((q \varphi')' \,  \varphi^2 \varphi' \, M^2 + 
   q \varphi' \, ( \varphi^2 \varphi')' \, M^2 + 
   q \varphi' \, \varphi' \varphi^2 \, 2 M M'\bigr)\,  q \varphi^2 \varphi' M  \\
&= B_0 C M  +   q  ( \varphi^2 \varphi')' \, C M + 2 q \varphi^2 \varphi'^2 M^2\, q \varphi^2 \varphi' M' \\
&=B_0 C M  +   q  ( \varphi^2 \varphi')' \, C M + 2 C (B- B_0 M)\\
&= C q  ( \varphi^2 \varphi')' \, M + 2 B C - B_0 C M.
\end{align*}
\end{proof}

\begin{lemma}\label{derivativehF} Let $I$, $I\subset (-\infty; +\infty)$, be a non-empty open interval. Assume the functions $q:I\to (0;+\infty)$, $\varphi:I\to (-\infty; +\infty)$, $M:I\to (0;+\infty)$ and $h:I\to (-\infty; +\infty)$ meet the conditions 
\begin{enumerate}\renewcommand{\labelenumi}{({\itshape \roman{enumi}}\,)~}
\item $q\in \smooth{2}{I}$, $\varphi \in \smooth{3}{I}$, $M\in \smooth{2}{I}$,
\item there exists a non empty open subinterval $J$ of $I$, $J\subseteq I$, such that $A(x)\neq 0$ for all $x\in J$ and $B^2-4AC\ge 0$ on $J$,
\item $h'=\varphi' M$ on $I$
\end{enumerate}
Then the following identities hold on $J$
\begin{align}
&(h-F_1)' A(F_2-F_1) q \varphi^2 \varphi' M \label{derivativehF1}\\
&\hphantom{(h-F_1} =F_1 \Bigl((F_1-\varphi M)(A^2 (F_1-\varphi M) + E_1 M)+C \varphi^2 \bigl( q\frac{M'}{M}\bigr)' \Bigr) \nonumber \\
&(h-F_2)' A(F_1-F_2) q \varphi^2 \varphi' M \label{derivativehF2}\\
&\hphantom{(h-F_2} = F_2 \Bigl((F_2-\varphi M)(A^2 (F_2-\varphi M) + E_1 M)+C \varphi^2 \bigl( q\frac{M'}{M}\bigr)' \Bigr) \nonumber
\end{align}
\end{lemma}

\begin{proof}[Proof of Lemma~\ref{derivativehF}] Each of these identities is a result of direct simple and rather long computations.  

The proof of identity~\eqref{derivativehF1} is as follows.
\begin{align}
&(h-F_1)' A(F_2-F_1) q \varphi^2 \varphi' M = (h'-F_1') A(F_2-F_1) q \varphi^2 \varphi' M \nonumber\\
&= \varphi'M A(F_2-F_1) q \varphi^2 \varphi' M  -F_1' A(F_2-F_1) q \varphi^2 \varphi' M.\label{derivativehF:eq1}
\end{align}
By the definition of the functions $F_{1,2}$ it follows that
\begin{align*}
A(F_2-F_1)&=\sqrt{B^2-4 A C}=-2 A F_1+B,\\
A F_1^2 - B F_1 +C&=0 \implies A' F_1^2 - B' F_1 +C'= F_1'(-2A F_1 + B) 
\end{align*}
and hence $F_1' A (F_2-F_1) = A' F_1^2 - B' F_1 +C'$. 

So, from~\eqref{derivativehF:eq1} it follows that
\begin{multline*}
(h-F_1)' A(F_2-F_1) q \varphi^2 \varphi' M
= A(F_2-F_1) C  \\
\shoveright{- (A' F_1^2-B' F_1+C') q \varphi^2 \varphi' M }\\
\shoveleft{=A(F_2-F_1) (-A F_1^2 +B F_1) - A' F_1^2 q \varphi^2 \varphi' M}\\
 + B' F_1 q \varphi^2 \varphi' M - C' q \varphi^2 \varphi' M
\end{multline*}
Now, identities~\eqref{derivativeB} and~\eqref{derivativeC} from Lemma~\ref{derivativesBC} allow us to obtain
\begin{multline*}
(h-F_1)' A(F_2-F_1) q \varphi^2 \varphi' M\\
\shoveleft{=-A^2 F_1^2 F_2 + A B F_1 F_2 + A^2 F_1^3 - A B F_1^2  - A' F_1^2 q \varphi^2 \varphi' M}\\
\shoveright{+B F_1  (B- B_0 M) + B_0' q \varphi^2 \varphi' M^2 F_1
 + (\varphi^2 \varphi')' q M (B- B_0 M) F_1 }\\
\shoveright{ + C \varphi^2  \bigl(q \frac{M'}{M}\bigr)' F_1 - C q  ( \varphi^2 \varphi')' \, M -2 B C + B_0 C M}\\
\shoveleft{=-A C F_1 + \underline{B C} + A^2 F_1^3 -\underline{ A B F_1^2} -  A' F_1^2 q \varphi^2 \varphi' M}\\
\shoveright{+ \underline{(A F_1^2+C)(B- B_0 M)} + B_0' q \varphi^2 \varphi' M^2 F_1 
  + \underline{(\varphi^2 \varphi')' q M B  F_1}} \\
\shoveright{- (\varphi^2 \varphi')' q B_0 M^2  F_1 + C \varphi^2  \bigl(q \frac{M'}{M}\bigr)' F_1 
- (- A F_1^2 +\underline{ B F_1}) q  ( \varphi^2 \varphi')' \, M}\\
 -\underline{2 B C} + \underline{B_0 C M}
\end{multline*}
where all the underlined parts cancel out. Thus
\begin{multline*}
(h-F_1)' A(F_2-F_1) q \varphi^2 \varphi' M\\
\shoveleft{=-A C F_1 + A^2 F_1^3 -  A' F_1^2 q \varphi^2 \varphi' M + B_0' q \varphi^2 \varphi' M^2 F_1 }\\
\shoveright{- (\varphi^2 \varphi')' q B_0 M^2  F_1 + C \varphi^2  \bigl(q \frac{M'}{M}\bigr)' F_1 
+ A F_1^2  q  ( \varphi^2 \varphi')' \, M}\\
\shoveleft{=F_1 \bigl( A^2 F_1^2 - (A B_0 - A q (\varphi^2 \varphi')' + A' q \varphi^2 \varphi' ) F_1 M }\\
- (A C_0 - B_0' q \varphi^2 \varphi' + (\varphi^2 \varphi')' q B_0 ) M^2  + C \varphi^2  \bigl(q \frac{M'}{M}\bigr)' \bigr)
\end{multline*}
Finally, by identities~\eqref{rootF2} we obtain
\begin{multline*}
(h-F_1)' A(F_2-F_1) q \varphi^2 \varphi' M\\
 =F_1 \Bigl((F_1-\varphi M)(A^2 F_1+ E_2 M)+C \varphi^2 \bigl( q\frac{M'}{M}\bigr)' \Bigr) \\
 =F_1 \Bigl((F_1-\varphi M)(A^2 (F_1-\varphi M) + E_1 M)+C \varphi^2 \bigl( q\frac{M'}{M}\bigr)' \Bigr) 
\end{multline*}

The computations that prove identity~\eqref{derivativehF2} are omitted as they are similar to those that prove identity~\eqref{derivativehF1}.
\end{proof}

\begin{lemma}\label{derivativehBC} Let $I$, $I\subset (-\infty; +\infty)$, be a non-empty open interval. Assume the functions $q:I\to (0;+\infty)$, $\varphi:I\to (-\infty; +\infty)$, $M:I\to (0;+\infty)$ and $h:I\to (-\infty; +\infty)$ meet the conditions 
\begin{enumerate}\renewcommand{\labelenumi}{({\itshape \roman{enumi}}\,)~}
\item $q\in \smooth{2}{I}$, $\varphi \in \smooth{3}{I}$, $M\in \smooth{2}{I}$,
\item there exists a non empty open subinterval $J$ of $I$, $J\subseteq I$, such that $A=0$ and $\varphi'>0$  on $J$ and $B(x)\neq 0$ for all $x\in J$, 
\item $h'=\varphi' M$ on $I$
\end{enumerate}
Then the following identity holds on $J$
\[
\bigl(h-\frac{C}{B}\bigr)'\, B^2 q \varphi' M = C^2 \bigl( q\frac{M'}{M}\bigr)'
\]
\end{lemma}

\begin{proof}[Proof of Lemma~\ref{derivativehBC}] \ \\
{\itshape {Claim.}\/} $(\frac{(q \varphi')'}{\varphi'})'=0$ on $J$ and $\varphi(x)\neq 0$ for all $x\in J$.

Indeed, note that $A=0$ implies $\varphi'' \varphi = \varphi'^2>0$.
So, $\varphi(x)\neq 0$ for all $x\in J$. Therefore, from $A=0$ follows
that $\frac{(q \varphi')'}{\varphi'}=\frac{q \varphi'}{\varphi}$ and
\[
\bigl(\frac{(q \varphi')'}{\varphi'}\bigr)'=\bigl(\frac{q \varphi'}{\varphi}\bigr)'=\frac{A}{\varphi^2}=0.
\]
Thus, the claim is proved.

Now, the lemma follows from the following computations
\begin{align*}
&\bigl( h - \frac{C}{B} \bigr)' B q \varphi^2 \varphi' M =
\bigl( h' - \frac{C'B - C B'}{B^2} \bigr)  B q \varphi^2 \varphi' M\\
&=\varphi'M B q \varphi^2 \varphi' M + B' \frac{C}{B} q \varphi^2 \varphi' M 
- C' q \varphi^2 \varphi' M
\end{align*}
$B'$ and $C'$ are substituted accordingly to identities~\eqref{derivativeB} and~\eqref{derivativeC} from Lemma~\ref{derivativesBC}
\begin{multline*}
\bigl( h - \frac{C}{B} \bigr)' B q \varphi^2 \varphi' M \\
\shoveleft{= \underline{B C} + \frac{C}{B} \Bigl( \underline{B  (B- B_0 M)} + B_0' q \varphi^2 \varphi' M^2 + (\varphi^2 \varphi')' q M (\underline{B}- B_0 M)}\\
   + C \varphi^2  \bigl(q \frac{M'}{M}\bigr)' \Bigr) -\underline{ (C q  ( \varphi^2 \varphi')' \, M + 2 B C - B_0 C M)}
\end{multline*}
where all the underlined parts cancel out. Thus,
\begin{align*}
&\bigl( h - \frac{C}{B} \bigr)' B q \varphi^2 \varphi' M \\
&= \frac{C}{B} \Bigl(  B_0' q \varphi^2 \varphi' M^2 - (\varphi^2 \varphi')' q M^2  B_0  + C \varphi^2  \bigl(q \frac{M'}{M}\bigr)' \Bigr)\\
&=\frac{C}{B} \Bigl(  q M^2 (\varphi^2 \varphi')^2 \bigl( \frac{B_0}{\varphi^2 \varphi'} \bigr)' + C \varphi^2  \bigl(q \frac{M'}{M}\bigr)' \Bigr)\\
&=\frac{C}{B} \Bigl(  q M^2 (\varphi^2 \varphi')^2 \bigl( \frac{(q \varphi')'}{\varphi'} \bigr)' + C \varphi^2  \bigl(q \frac{M'}{M}\bigr)' \Bigr)
\end{align*}
Finally, accordingly to the claim, it follows that
\begin{align*}
&\bigl( h - \frac{C}{B} \bigr)' B q \varphi^2 \varphi' M = \frac{C^2}{B} \varphi^2  \bigl(q \frac{M'}{M}\bigr)'.
\end{align*}
\end{proof}

\section{Main results}
\begin{theorem}\label{th1}
Assume the functions $\varphi:(0;+\infty)\to (-\infty; +\infty)$, $M:(0;+\infty)\to (0;+\infty)$ and $h:(0;+\infty)\to (-\infty; +\infty)$ are such that $\varphi \in \smooth{3}{0;+\infty}$, $M\in \smooth{2}{0;+\infty}$ and meet the conditions 
\begin{enumerate}\renewcommand{\labelenumi}{({\itshape \roman{enumi}}\,)~}
\item $M'<0$ and $(\log M)''\ge 0$ on $(0;+\infty)$ 
\item $\varphi'>0$ on $(0;+\infty)$ and $\varphi(x)=\int_1^x \varphi'(t) dt$
for all $x\in (0;+\infty)$
\item $\varphi'^2\varphi''+\varphi \varphi' \varphi''' - 2 \varphi \varphi''^2 \le 0$ on $(0;+\infty)$, 
\item $h(x)=\int_1^x \varphi'(t) M(t) dt$ for all $x\in(0;+\infty)$
\end{enumerate}
Then $\frac{h}{\varphi}$ and  $\log \frac{h}{\varphi}$ both belong to $\smooth{2}{0;+\infty}$ and moreover
\[
\bigl(\frac{h}{\varphi} \bigr)'<0, \bigl(\log \frac{h}{\varphi} \bigr)''>0 \mbox{ on } (0;+\infty)
\]
\end{theorem}

\begin{proof}[Proof of Theorem~\ref{th1}] \ \\
{\itshape {Claim.}\/} $\tfrac{h}{\varphi}\in \smooth{2}{0;+\infty}$.

Indeed, accordingly to the assumptions, it is clear that the functions $h$ and $\varphi$ belong to $\smooth{3}{0;+\infty}$. Moreover,  $\varphi(x)=0$ $\iff$ $x=1$. So, it is sufficient to prove that $\tfrac{h}{\varphi}$ has asymptotic expansion of the form
\[
\tfrac{h(x)}{\varphi(x)}=\alpha_0+\alpha_1 (x-1) + \alpha_2 (x-1)^2 +o(x-1)^2,
\]
as $x\to 1$, where $\alpha_{0,1,2}$ are real numbers that do not depend on $x$.

The expansion is obtained as follows. Note that $h(1)=\varphi(1)=0$ and
\begin{align}
h(x)&=h'(1)(x\!-\!1)+\tfrac12 h''(1)(x\!-\!1)^2 +\tfrac16 h'''(1)(x\!-\!1)^3 \!+\! o(x\!-\!1)^3 \label{th1:h}\\
\varphi(x)&=\varphi'(1)(x\!-\!1)+\tfrac12 \varphi''(1)(x\!-\!1)^2 + \tfrac16 \varphi'''(1)(x\!-\!1)^3\!+\! o(x\!-\!1)^3\label{th1:f}
\end{align}
as $x\to 1$ with $\varphi'(1)\neq 0$ (by the assumptions of the theorem). Therefore, 
\[
\tfrac{h(x)}{\varphi(x)}=\beta_0+\beta_1 (x-1) + \beta_2 (x-1)^2 + o(x-1)^2
\]
where $\beta_0=\tfrac{h'(1)}{\varphi'(1)}$, $\beta_1=\tfrac12 (h''(1)-\varphi''(1)\beta_0)\tfrac{1}{\varphi'(1)}$ and
\[
\beta_2=\tfrac16 (h'''(1)-3 \beta_1 \varphi''(1)- \beta_0 \varphi'''(1)) \tfrac{1}{\varphi'(1)}
\]
So,
\[
\tfrac{h(x)}{\varphi(x)}=M(1)+\tfrac12 M'(1) (x-1) + \tfrac16 ( M''(1) + \tfrac{\varphi''(1)}{2\varphi'(1)}M'(1) ) (x-1)^2 +o(x-1)^2,
\]
as $x\to 1$ and the claim is proved.

Let us define the value of $\tfrac{h}{\varphi}$ at $x=1$ to be equal to
\[
\lim_{x\to 1}\tfrac{h(x)}{\varphi(x)}=M(1)
\]
and note that $M(1)>0$. Moreover, 
\[
\left.(\tfrac{h(x)}{\varphi(x)})'\right|_{x=1}\,= \tfrac12 M'(1),\quad
\left.(\tfrac{h(x)}{\varphi(x)})''\right|_{x=1}\,= \tfrac13 ( M''(1) + \tfrac{\varphi''(1)}{2\varphi'(1)}M'(1) )
\]
where $\left.  \right|_{x=1} $ stands for `the value at $x=1$'. 

Now, it follows from the claim and from $\tfrac{h}{\varphi}>0$ on $(0;+\infty)$ that
$\log \frac{h}{\varphi}$ is well defined on $(0;+\infty)$ and belongs to 
$\smooth{2}{0;+\infty}$.

Here, it is verified that the functions $h$, $\varphi$ and $M$ satisfy
an important simple inequality. 

{\itshape {Claim.}\/} $h-\varphi M >0$ on $(0;1)\cup(1;+\infty)$.

This inequality holds because of
\[
h(x)-\varphi(x)M(x)=(-1)\int_1^x \varphi(t) M'(t) dt >0
\]
for all $x\in (0;1)\cup(1;+\infty)$.

Hence, the derivative
\[
( \tfrac{h}{\varphi})'=(-1)\tfrac{\varphi'}{\varphi^2}(h-\varphi M)<0,
\mbox{ on } (0;+\infty).
\]

The derivative $(\log \tfrac{h}{\varphi} )''$ is calculated as follows.

From this point of the proof of the theorem to its end let the function $q(x)=1$ for all $x\in(0;+\infty)$ and $A$, $B$, $C$ be as in Definition~\ref{defABC}. 

In particular, 
\begin{gather*}
A= \varphi'' \varphi -  \varphi'^2,\quad B=( \varphi' M)' \varphi^2, \quad C=\varphi^2 \varphi'^2 M^2.
\end{gather*} 
and by~\eqref{E1identity} 
\[
E_1=-\varphi^2 \varphi'^3 \bigl( \frac{\varphi'' \varphi}{ \varphi'^2} \bigr)'=-\varphi^2(\varphi'^2\varphi''+\varphi \varphi' \varphi''' - 2 \varphi \varphi''^2 ).
\]

Thus, accordingly to the assumptions of the theorem, $E_1\ge 0$ on $(0;1)\cup(1;+\infty)$ and $\left. E_1 \right|_{x=1}\, = 0$. 
Furthermore,
\begin{align*}
A(1) &= \varphi''(1) \varphi(1) -  \varphi'^2(1)=
-  \varphi'^2(1)< 0\\
\bigl(\tfrac{A}{\varphi'^2}\bigr)'&=\tfrac{\varphi \varphi' \varphi''' + \varphi'^2 \varphi'' - 2 \varphi \varphi''^2}{\varphi'^3}\le 0 \mbox{ on } (0;+\infty)
\end{align*}
So, $\tfrac{A}{\varphi'^2}=\frac{\varphi'' \varphi}{\varphi'^2}-1$ decreases on $(0;+\infty)$ and therefore 
\begin{itemize}
\item $\left. \tfrac{A}{\varphi'^2} \right|_{x=1}\, <0$ $\implies$
there exist two possible cases: $A<0$ on $(0;+\infty)$ or there exists
a number $x_A\in(0;1)$ such that $A>0$ on $(0; x_A)$, $A(x_A)=0$ and $A<0$ on $(x_A;+\infty)$ (these cases are discussed bellow as {\itshape {Case 1}} and {\itshape {Case 2}});
\item $\frac{\varphi'' \varphi}{\varphi'^2}$ decreases on $(0;+\infty)$, its value $\left.\frac{\varphi'' \varphi}{\varphi'^2}\right|_{x=1}\, =0$ and hence $\varphi''<0$, $B=( \varphi'' M+ \varphi' M') \varphi^2<0$ on $(0;1)\cup(1;+\infty)$.
\end{itemize}

{\itshape {Case 1.}\/} Let us suppose that the function $\varphi$ is such that $A<0$ on $(0;+\infty)$.

In this case, $B^2-4AC>0$ on $(0;1)\cup(1;+\infty)$. So, the functions $F_{1,2}$ are well defined real valued on $(0;1)\cup(1;+\infty)$ and
\[
F_2<0<\tfrac{B}{2A}<F_1.
\]
as $C>0$. Moreover, 
\[
\left. (A F^2-B F+ C)\right|_{F=\varphi M}\, = - \varphi^3 \varphi' M M' \begin{cases}
<0& x\in(0;1)\\
>0& x\in(1;+\infty)
\end{cases}
\]
Therefore,
\begin{itemize}
\item if $x\in(0;1)$ then $\varphi M< F_2<0<F_1$, $\varphi M < h<0$ and by the identity~\eqref{derivativehF2} from Lemma~\ref{derivativehF} we obtain
\[
\begin{split}
&(h-F_2)' A(F_1-F_2) q \varphi^2 \varphi' M\\
&\hphantom{(h-F_2} = F_2 \Bigl((F_2-\varphi M)(A^2 (F_2-\varphi M) + E_1 M)+C \varphi^2 \bigl( q\frac{M'}{M}\bigr)' \Bigr)<0
\end{split}
\]
So, $(h-F_2)'>0$ on $(0;1)$. By the definition of $F_2$ it follows that the left hand side limit 
\[
\lim_{x\to 1^{-}}(h(x)-F_2(x))=0
\]
Thus, $h-F_2<0$ on $(0;1)$ and hence 
\[
A h^2-B h+C <0 \mbox{ on } (0;1).
\]

\item if $x\in(1;+\infty)$ then $ F_2<0<\varphi M<F_1$, $\varphi M \le h$ and by the identity~\eqref{derivativehF1} from Lemma~\ref{derivativehF} we obtain
\[
\begin{split}
&(h-F_1)' A(F_2-F_1) q \varphi^2 \varphi' M\\
&\hphantom{(h-F_1} = F_1 \Bigl((F_1-\varphi M)(A^2 (F_1-\varphi M) + E_1 M)+C \varphi^2 \bigl( q\frac{M'}{M}\bigr)' \Bigr)>0
\end{split}
\]
So, $(h-F_1)'>0$ on $(1;+\infty)$. By the definition of $F_1$ it follows that 
the right hand side limit 
\[
\lim_{x\to 1^{+}}(h(x)-F_1(x))=0
\]
Thus, $h-F_1>0$ on $(1;+\infty)$ and hence 
\[
A h^2-B h+C <0 \mbox{ on } (1;+\infty).
\]
\end{itemize}

Hence, in case 1,  $A h^2-B h+C <0$  on  $(0;1)\cup(1;+\infty)$.

{\itshape {Case 2.}\/} Let us suppose that the function $\varphi$ is such that there exists a number $x_A\in(0;1)$ such that $A>0$ on $(0; x_A)$, $A(x_A)=0$ and $A<0$ on $(x_A;+\infty)$.

In this case, 
\begin{itemize}
\item if $x\in(0;x_A)$ then $\varphi M \le h <0$,
\[
\left. (A F^2-B F+ C)\right|_{F=\varphi M}\, = - \varphi^3 \varphi' M M'<0,
\]
$C>0$ and hence $F_{1,2}$ are well defined real valued functions such that
\[
F_1<\varphi M < F_2 <0.
\]
We apply the identity~\eqref{derivativehF2} from Lemma~\ref{derivativehF} to obtain
\[
\begin{split}
&(h-F_2)' A(F_1-F_2) q \varphi^2 \varphi' M\\
&\hphantom{(h-F_2} = F_2 \Bigl((F_2-\varphi M)(A^2 (F_2-\varphi M) + E_1 M)+C \varphi^2 \bigl( q\frac{M'}{M}\bigr)' \Bigr)<0
\end{split}
\]
So, $(h-F_2)'>0$ on $(0;x_A)$.

\item if $x\in(x_A;1)$ then $A<0$, $C>0$, $B^2-4AC>0$,  $\varphi M \le h <0$,
\[
\left. (A F^2-B F+ C)\right|_{F=\varphi M}\, = - \varphi^3 \varphi' M M'<0,
\]
and hence $F_{1,2}$ are well defined real valued functions such that
\[
\varphi M <F_2< 0< F_1
\]
We apply the identity~\eqref{derivativehF2} from Lemma~\ref{derivativehF} to obtain
\[
\begin{split}
&(h-F_2)' A(F_1-F_2) q \varphi^2 \varphi' M\\
&\hphantom{(h-F_2} = F_2 \Bigl((F_2-\varphi M)(A^2 (F_2-\varphi M) + E_1 M)+C \varphi^2 \bigl( q\frac{M'}{M}\bigr)' \Bigr)<0
\end{split}
\]
So, $(h-F_2)'>0$ on $(x_A;1)$.
\end{itemize}

Thus, $(h-F_2)'>0$ on $(0;x_A)\cup(x_A;1)$. By $F_2=\frac{2C}{B-\sqrt{B^2-4AC}}$
it follows that $F_2$ is continuous at $x=x_A$. So, $h-F_2$ increases on $(0;1)$ and by $\lim\limits_{x\to 1^{-}}(h-F_2)=0$ we obtain that
\[
h-F_2<0 \mbox{ on } (0;1).
\]
Hence 
\[
A h^2-B h+C <0 \mbox{ on } (0;1).
\]

\begin{itemize}
\item if $x\in(1;+\infty)$ then $A<0$, $C>0$, $B^2-4AC>0$,  $0<\varphi M \le h$,
\[
\left. (A F^2-B F+ C)\right|_{F=\varphi M}\, = - \varphi^3 \varphi' M M'>0,
\]
and hence $F_{1,2}$ are well defined real valued functions such that
\[
F_2< 0<\varphi M <F_1
\]
We apply the identity~\eqref{derivativehF1} from Lemma~\ref{derivativehF} to obtain
\[
\begin{split}
&(h-F_1)' A(F_2-F_1) q \varphi^2 \varphi' M\\
&\hphantom{(h-F_1} = F_1 \Bigl((F_1-\varphi M)(A^2 (F_1-\varphi M) + E_1 M)+C \varphi^2 \bigl( q\frac{M'}{M}\bigr)' \Bigr)>0
\end{split}
\]
So, $(h-F_1)'>0$ on $(1;+\infty)$.
By the definition of $F_1$ it follows that 
the right hand side limit 
\[
\lim_{x\to 1^{+}}(h(x)-F_1(x))=0
\]
Thus, $h-F_1>0$ on $(1;+\infty)$ and hence 
\[
A h^2-B h+C <0 \mbox{ on } (1;+\infty).
\]
\end{itemize} 

Hence, in case 2,  $A h^2-B h+C <0$  on  $(0;1)\cup(1;+\infty)$.

Therefore 
\begin{equation}\label{eq:negative}
A h^2-B h+C <0 \mbox{ on } (0;1)\cup(1;+\infty)
\end{equation} 
in both {\itshape {Case 1.}\/} and~{\itshape {Case 2.}\/}

{\itshape {Claim.}\/} Let the function $q(x)=1$ for all $x\in(0;+\infty)$ and $A$, $B$, $C$ be as in Definition~\ref{defABC}. Then 
\begin{itemize}
\item there exists the limit
$\lim\limits_{x\to 1} \tfrac{(-1)(A h^2 - B h + C)}{\varphi^2 h^2} >0$
\item the second derivative 
$(\log \tfrac{h}{\varphi} )''=\tfrac{-1}{\varphi^2 h^2} (A h^2 - B h + C)$ 
on $(0;+\infty)$
\end{itemize}
In particular, $(\log \tfrac{h}{\varphi} )''$ is a well defined continuous function on $(0;+\infty)$.

The proof of the first item of this Claim is sketched out only. The limit is computed by using expansions~\eqref{th1:h},~\eqref{th1:f} and
$M(x)=M(1)+M'(1)(x-1)+\frac12 M''(1)(x-1)^2+o(x-1)^2$ as $x\to 1$
and the values of the derivatives of $h$ are substituted as follows $h'(1)=\varphi'(1) M(1)$, $h''(1)=\varphi''(1) M(1) + \varphi'(1) M'(1)$, $h'''(1)=\varphi'''(1) M(1) + 2\varphi''(1) M'(1)+ \varphi'(1) M''(1)$. 

The elements of the numerator, $Ah^2$, $Bh$ and $C$, are calculated with a
precision of $o(x-1)^4$. So, the numerator 
\begin{multline*}
(-1)(A(x) h(x)^2-B(x) h(x)+C(x))\\
=\tfrac{\varphi'^4(1)M^2(1)}{6}(\tfrac{M'(1)}{M(1)}\,\tfrac{\varphi''(1)}{\varphi'(1)}
+  \tfrac{4 M(1)M''(1) -3 M'^2(1)}{2\, M^2(1)})(x-1)^4+o(x-1)^4
\end{multline*}
as $x\to 1$.

The denominator, $\varphi^2 h^2$, is  calculated with a
precision of $o(x-1)^4$. So, the denominator 
\[
\varphi^2 h^2=M^2(1)\varphi'^4(1)(x-1)^4+o(x-1)^4
\]
as $x\to 1$.

Thus, there exists the limit
\[
\lim_{x\to 1} \tfrac{(-1)(A h^2 - B h + C)}{\varphi^2 h^2} =
\tfrac{1}{6}(\tfrac{M'(1)}{M(1)}\,\tfrac{\varphi''(1)}{\varphi'(1)}
+  \tfrac{4 M(1)M''(1) -3 M'^2(1)}{2\, M^2(1)})
\]
Note that the limit is a positive number because 
\begin{gather*}
M'(1)\varphi''(1)\ge 0,\\
4 M(1)M''(1) -3 M'^2(1)= M(1)M''(1)+3( M(1)M''(1) - M'^2(1))>0.
\end{gather*} 

In order to prove the second item of the Claim, 
we calculate the second derivative 
\[
(\log \tfrac{h}{\varphi} )''=(q (\log \tfrac{h}{\varphi} )')'=
(\tfrac{q h' \varphi}{h \varphi'})'=(\tfrac{q \varphi M}{h })'=
\tfrac{(-1)(A h^2 - B h + C)}{\varphi^2 h^2} 
\]
on $(0;1)\cup(1;+\infty)$ and it is a continuous on $(0;1)$ and $(1;+\infty)$. There exists a finite limit of the second derivative as $x\to 1$, accordingly to the first item. 
This result and the existence of the finite derivative $(\log \tfrac{h}{\varphi} )''$ at $x=1$ (note that $\log \tfrac{h}{\varphi}\in \smooth{2}{0;+\infty}$), they imply that the second derivative $(\log \frac{h}{\varphi} )''$ is a continuous function at $x=1$.
Therefore, the second derivative $(\log \frac{h}{\varphi} )''$ is a continuous function on $(0;+\infty)$.

So, the Claim is proved. 

By this Claim and~\eqref{eq:negative} it follows that the second derivative 
\[
(\log \frac{h}{\varphi} )''>0 \mbox{ on } (0;+\infty).
\]  
Thus, the theorem is proved. 
\end{proof}

\begin{theorem}\label{th2}
Assume the functions $\varphi:[0;+\infty)\to [0; +\infty)$, $M:[0;+\infty)\to (0;+\infty)$ and $h:[0;+\infty)\to [0; +\infty)$ are such that $\varphi \in \smooth{3}{0;+\infty}$, $M\in \smooth{2}{0;+\infty}$ and meet the conditions 
\begin{enumerate}\renewcommand{\labelenumi}{({\itshape \roman{enumi}}\,)~}
\item the right hand limit $\lim\limits_{x\to 0^{+}}\frac{\varphi''(x) \varphi(x)}{\varphi'^2(x)}<1$,
\item the functions $M$ and $M'$ are continuous from the right at $x=0$ and  $M(0)\neq +\infty$ and $M'(0)\neq -\infty$,
\item $M'<0$, $(\log M)''\ge 0$ on $(0;+\infty)$ 
\item $\varphi'>0$ on $(0;+\infty)$ and $\varphi(x)=\int_0^x \varphi'(t) dt$
for all $x\in (0;+\infty)$
\item $\varphi'^2\varphi''+\varphi \varphi' \varphi''' - 2 \varphi \varphi''^2 \le 0$ on $(0;+\infty)$,
\item $h(x)=\int_0^x \varphi'(t) M(t) dt$ for all $x\in(0;+\infty)$
\end{enumerate}
Then $\frac{h}{\varphi}$ and  $\log \frac{h}{\varphi}$ both belong to $\smooth{2}{0;+\infty}$ and moreover
\[
\bigl(\frac{h}{\varphi} \bigr)'<0, \bigl(\log \frac{h}{\varphi} \bigr)''>0 \mbox{ on } (0;+\infty)
\]
\end{theorem}

\begin{proof}[Proof of Theorem~\ref{th2}]
By the assumptions it is clear that $\frac{h}{\varphi}$ and  $\log \frac{h}{\varphi}$ both belong to $\smooth{2}{0;+\infty}$. Furthermore,
\[
h(x)-\varphi(x) M(x) = \int_0^x (-1) \varphi(t) M'(t) dt>0 \mbox{ for all } x\in(0;+\infty).
\]
Hence, 
\[
(\tfrac{h}{\varphi})'=\tfrac{-\varphi'}{\varphi^2} (h- \varphi M) <0 \mbox{ on }
(0;+\infty).
\]

From this point of the proof of the theorem to its end let the function $q(x)=1$ for all $x\in(0;+\infty)$ and $A$, $B$, $C$ be as in Definition~\ref{defABC}. 

In particular, 
\begin{gather*}
A= \varphi'' \varphi -  \varphi'^2,\quad B=( \varphi' M)' \varphi^2, \quad C=\varphi^2 \varphi'^2 M^2.
\end{gather*} 
and by~\eqref{E1identity} 
\[
E_1=-\varphi^2 \varphi'^3 \bigl( \frac{\varphi'' \varphi}{ \varphi'^2} \bigr)'=-\varphi^2(\varphi'^2\varphi''+\varphi \varphi' \varphi''' - 2 \varphi \varphi''^2 ).
\]

Thus, accordingly to the assumptions of the theorem, $E_1\ge 0$ on $(0;+\infty)$. 
Furthermore,
\[
\bigl(\tfrac{A}{\varphi'^2}\bigr)'=\tfrac{\varphi \varphi' \varphi''' + \varphi'^2 \varphi'' - 2 \varphi \varphi''^2}{\varphi'^3}\le 0 \mbox{ on } (0;+\infty)
\]
So, $\tfrac{A}{\varphi'^2}=\frac{\varphi'' \varphi}{\varphi'^2}-1$ decreases on $(0;+\infty)$ and $\tfrac{A}{\varphi'^2}\le \lim_{x\to 0^{+}} \tfrac{A}{\varphi'^2}<0$  

The inequalities $A<0$, $C>0$ on $(0;+\infty)$ imply that $F_1$, $F_2$
are well defined real valued functions on $(0;+\infty)$ and together with the inequality 
\[
\left. (A F^2-B F+ C)\right|_{F=\varphi M}\, = - \varphi^3 \varphi' M M' >0
\]
it follows that
\[
F_2<0< \varphi M< F_1 \mbox{ on } (0;+\infty).
\]

We apply the identity~\eqref{derivativehF1} from Lemma~\ref{derivativehF} to obtain
\[
\begin{split}
&(h-F_1)' A(F_2-F_1) q \varphi^2 \varphi' M\\
&\hphantom{(h-F_1} = F_1 \Bigl((F_1-\varphi M)(A^2 (F_1-\varphi M) + E_1 M)+C \varphi^2 \bigl( q\frac{M'}{M}\bigr)' \Bigr)>0
\end{split}
\]
So, $(h-F_1)'>0$ on $(0;+\infty)$.

Moreover, $\left. (h-F_1)\right|_{0^{+}} = 0$ because of $h(0^{+})=0 $ and 
\begin{align*}
&\lim_{x\to 0^{+}} F_1(x) = \lim_{x\to 0^{+}} \frac{B-\sqrt{B^2-4AC}}{2A} \\
&=\lim_{x\to 0^{+}} \tfrac{(\varphi''M+\varphi'M')\varphi^2- \sqrt{(\varphi''M+\varphi'M')^2\varphi^4- 4 A \varphi^2 \varphi'^2 M^2}}{2A}\\
&=\lim_{x\to 0^{+}} \varphi \,\tfrac{\bigl( (\frac{A}{\varphi'^2}+1)M + \frac{\varphi}{\varphi'} M'\bigr) - \sqrt{ \bigl( (\frac{A}{\varphi'^2}+1)M + \frac{\varphi}{\varphi'} M'\bigr)^2 - 4 \frac{A}{\varphi'^2} M^2 } }{ 2 \frac{A}{\varphi'^2} }=0
\end{align*}

Hence, $h-F_1>0$ on $(0;+\infty)$. Therefore,
\[
A h^2-B h+C <0 \mbox{ on } (0;+\infty)
\]
and the second derivative 
\[
(\log \tfrac{h}{\varphi} )''=(q (\log \tfrac{h}{\varphi} )')'=
(\tfrac{q h' \varphi}{h \varphi'})'=(\tfrac{q \varphi M}{h })'=
\tfrac{(-1)(A h^2 - B h + C)}{\varphi^2 h^2} >0
\]
on $(0;+\infty)$.

Thus the theorem is proved.
\end{proof}

\begin{theorem}\label{th3}
Assume the functions $\varphi:(0;+\infty)\to (-\infty; 0)$, $M:(0;+\infty)\to (0;+\infty)$ and $h:(0;+\infty)\to (-\infty;0)$ are such that $\varphi \in \smooth{3}{0;+\infty}$, $M\in \smooth{2}{0;+\infty}$ and meet the conditions 
\begin{enumerate}\renewcommand{\labelenumi}{({\itshape \roman{enumi}}\,)~}
\item $M'<0$ and $(\log M)''\ge 0$ on $(0;+\infty)$ 
\item $\varphi'>0$ on $(0;+\infty)$ and $\varphi(x)=-\int_x^{+\infty} \varphi'(t) dt$, $\forall x\in (0;+\infty)$
\item $\varphi'^2\varphi''+\varphi \varphi' \varphi''' - 2 \varphi \varphi''^2 \le 0$ on $(0;+\infty)$, 
\item $h(x)=-\int_x^{+\infty} \varphi'(t) M(t) dt$ for all $x\in(0;+\infty)$
\end{enumerate}
Then $\frac{h}{\varphi}$ and  $\log \frac{h}{\varphi}$ both belong to $\smooth{2}{0;+\infty}$ and moreover
\[
\bigl(\frac{h}{\varphi} \bigr)'<0, \bigl(\log \frac{h}{\varphi} \bigr)''>0 \mbox{ on } (0;+\infty)
\]
\end{theorem}

\begin{proof}[Proof of Theorem~\ref{th3}]
By the assumptions it is clear that $\frac{h}{\varphi}$ and  $\log \frac{h}{\varphi}$ both belong to $\smooth{2}{0;+\infty}$. 

{\itshape {Claim.}\/} $h-\varphi M>0$ on $(0;+\infty)$.
Indeed, by the assumptions of the theorem 
\begin{itemize}
\item $(h-\varphi M)'=- \varphi M' <0$ on $(0;+\infty)$
\item $M$ decreases on $(0;+\infty)$ and $M>0$. So, the \\ limit $\lim_{x\to{+\infty}} M(x)$ exists and it is a non-negative number.
\end{itemize}
Therefore,  $h-\varphi M$ decreases on $(0;+\infty)$ and 
\[
\lim_{x\to +\infty} (h(x)-\varphi(x) M(x)) = 0.
\]
Hence, $h-\varphi M>0$ on $(0;+\infty)$.

It follows from this Claim that the derivative
\[
(\tfrac{h}{\varphi})'=\tfrac{-\varphi'}{\varphi^2} (h- \varphi M) <0 \mbox{ on }
(0;+\infty).
\]

From this point of the proof of the theorem to its end let the function $q(x)=1$ for all $x\in(0;+\infty)$ and $A$, $B$, $C$ be as in Definition~\ref{defABC}. 

In particular, 
\begin{gather*}
A= \varphi'' \varphi -  \varphi'^2,\quad B=( \varphi' M)' \varphi^2, \quad C=\varphi^2 \varphi'^2 M^2.
\end{gather*} 
and by~\eqref{E1identity} 
\[
E_1=-\varphi^2 \varphi'^3 \bigl( \frac{\varphi'' \varphi}{ \varphi'^2} \bigr)'=-\varphi^2(\varphi'^2\varphi''+\varphi \varphi' \varphi''' - 2 \varphi \varphi''^2 ).
\]

Note that for every $x\in(0;+\infty)$
\begin{equation}\label{th3:negativity}
\left. (A F^2-B F+ C)\right|_{F=\varphi M}\, = - \varphi^3 \varphi' M M' <0.
\end{equation}

Thus, accordingly to the assumptions of the theorem, $E_1\ge 0$ on $(0;+\infty)$. 
Furthermore,
\[
\bigl(\tfrac{A}{\varphi'^2}\bigr)'=\tfrac{\varphi \varphi' \varphi''' + \varphi'^2 \varphi'' - 2 \varphi \varphi''^2}{\varphi'^3}\le 0 
\]
and $\tfrac{A}{\varphi'^2}$ decreases on $(0;+\infty)$.

Now, there are three cases.

{\itshape {Case 1.}\/} $A<0$ on $(0;+\infty)$.
In this case $B^2-4AC>0$ and hence $F_{1,2}$ are well defined real valued functions such that
\[
\varphi M < F_2 <0<F_1.
\]
Hence, $0=\lim_{x\to +\infty} \varphi(x) M(x) \le\lim_{x\to +\infty} F_2(x)\le 0$ and 
\[
\lim_{x\to +\infty} (h(x)-F_2(x))=0.
\]

We apply the identity~\eqref{derivativehF2} from Lemma~\ref{derivativehF} to obtain
\[
\begin{split}
&(h-F_2)' A(F_1-F_2) q \varphi^2 \varphi' M\\
&\hphantom{(h-F_2} = F_2 \Bigl((F_2-\varphi M)(A^2 (F_2-\varphi M) + E_1 M)+C \varphi^2 \bigl( q\frac{M'}{M}\bigr)' \Bigr)<0
\end{split}
\]
So, $(h-F_2)'>0$ on $(0;+\infty)$ and $h-F_2$ increases on $(0;+\infty)$.

Therefore, $h-F_2<0$ on $(0;+\infty)$ and by $\varphi M <h< F_2 <0<F_1$ it follows that 
\[
Ah^2-Bh+C<0 \mbox{ on } (0;+\infty).
\]

{\itshape {Case 2.}\/} There exists $x_A\in(0;+\infty)$ such that  $A>0$ on $(0;x_A)$,  $A(x_A)=0$, $A<0$ on $(x_A;+\infty)$.

In this case, 
\begin{itemize}
\item if $x\in(0;x_A)$ then 
\[
\left. (A F^2-B F+ C)\right|_{F=\varphi M}\, = - \varphi^3 \varphi' M M'<0,
\]
$C>0$ and hence $F_{1,2}$ are well defined real valued functions such that
\[
F_1<\varphi M < F_2 <0.
\]
We apply the identity~\eqref{derivativehF2} from Lemma~\ref{derivativehF} to obtain
\[
\begin{split}
&(h-F_2)' A(F_1-F_2) q \varphi^2 \varphi' M\\
&\hphantom{(h-F_2} = F_2 \Bigl((F_2-\varphi M)(A^2 (F_2-\varphi M) + E_1 M)+C \varphi^2 \bigl( q\frac{M'}{M}\bigr)' \Bigr)<0
\end{split}
\]
So, $(h-F_2)'>0$ on $(0;x_A)$.
\item if $x\in(x_A;+\infty)$ then $B^2-4AC>0$ and hence $F_{1,2}$ are well defined real valued functions such that
\[
\varphi M < F_2 <0<F_1.
\]
Hence, $0=\lim_{x\to +\infty} \varphi(x) M(x) \le\lim_{x\to +\infty} F_2(x)\le 0$ and 
\[
\lim_{x\to +\infty} (h(x)-F_2(x))=0.
\]
We apply the identity~\eqref{derivativehF2} from Lemma~\ref{derivativehF} to obtain
\[
\begin{split}
&(h-F_2)' A(F_1-F_2) q \varphi^2 \varphi' M\\
&\hphantom{(h-F_2} = F_2 \Bigl((F_2-\varphi M)(A^2 (F_2-\varphi M) + E_1 M)+C \varphi^2 \bigl( q\frac{M'}{M}\bigr)' \Bigr)<0
\end{split}
\]
So, $(h-F_2)'>0$ on $(x_A;+\infty)$.
\end{itemize}
Thus, $h-F_2$ increases on $(0;x_A)$ and on $(x_A;+\infty)$.

Moreover, in this case, $\varphi''(x_A)<0$. Hence,  the inequality $B=(\varphi'' M + \varphi' M' )\varphi^2<0$ holds in a neighborhood of $x_A$. So,
\[
F_2=\frac{2C}{B-\sqrt{B^2-4AC}} 
\]
is continuous.

Hence,  $h-F_2$ increases on $(0;+\infty)$.
Therefore, $h-F_2<0$ on $(0;+\infty)$.

Now, we prove that $A h^2-B h +C <0$ on $(0;+\infty)$.
Indeed, 
\begin{itemize}
\item $Ah^2-Bh+C<0$ on $(0;x_A)$ because of $F_1<\varphi M < h < F_2<0$ and $A>0$
\item $Ah^2-Bh+C<0$ on $(0;x_A)$ because of $\varphi M < h < F_2<0<F_1$ and $A<0$
\end{itemize}

{\itshape {Case 3.}\/} $A>0$ on $(0;+\infty)$.
In this case, 
\[
\left. (A F^2-B F+ C)\right|_{F=\varphi M}\, = - \varphi^3 \varphi' M M'<0,
\]
$C>0$ and hence $F_{1,2}$ are well defined real valued functions such that
\[
F_1<\varphi M < F_2 <0.
\]
Hence, $0=\lim_{x\to +\infty} \varphi(x) M(x) \le\lim_{x\to +\infty} F_2(x)\le 0$ and 
\[
\lim_{x\to +\infty} (h(x)-F_2(x))=0.
\]

We apply the identity~\eqref{derivativehF2} from Lemma~\ref{derivativehF} to obtain
\[
\begin{split}
&(h-F_2)' A(F_1-F_2) q \varphi^2 \varphi' M\\
&\hphantom{(h-F_2} = F_2 \Bigl((F_2-\varphi M)(A^2 (F_2-\varphi M) + E_1 M)+C \varphi^2 \bigl( q\frac{M'}{M}\bigr)' \Bigr)<0
\end{split}
\]
So, $(h-F_2)'>0$ on $(0;+\infty)$.
Hence,  $h-F_2$ increases on $(0;+\infty)$.
Therefore, $h-F_2<0$ on $(0;+\infty)$ and
\[
Ah^2-Bh+C<0 \mbox{ on } (0;+\infty).
\]

Therefore,
\[
A h^2-B h+C <0 \mbox{ on } (0;+\infty)
\]
and the second derivative 
\[
(\log \tfrac{h}{\varphi} )''=(q (\log \tfrac{h}{\varphi} )')'=
(\tfrac{q h' \varphi}{h \varphi'})'=(\tfrac{q \varphi M}{h })'=
\tfrac{(-1)(A h^2 - B h + C)}{\varphi^2 h^2} >0
\]
on $(0;+\infty)$.

Thus the theorem is proved.
\end{proof}

\section{Applications}
Here we start with a note about integral means of holomorphic on the upper half plane functions proved in a paper~\cite{HardyInghamPolya1927} by G.~Hardy, A.~Ingham, G.~Pol\'ya in 1927.
In subsection~4.3 they proved that if the holomorphic function in a strip of the complex plane meets some conditions such as growth at infinity $O(e^{e^{k|z|}})$ and convergence
of the integrals on the boundaries of the strip then in the case when $2\le p <+\infty$ the integral mean
\[
M(y)=\int_{-\infty}^{+\infty} |f(x+iy)|^p dx
\] 
has first derivative
\[
M'(y)=p\int_{-\infty}^{+\infty} |f(x+iy)|^{p-2}(u u'_y+ v v'_y) dx
\]
and second derivative
\[
M''(y)=p^2\int_{-\infty}^{+\infty} |f(x+iy)|^{p-2}(u'^2_y+ v'^2_y) dx
\]
where $u$ is the real part of $f$ and $v$ is the imaginary part of $f$.
Moreover, they proved that 
\[
M'^2\le M'' M.
\]
So, $(\log M)''\ge 0$. Note that, $M''\ge 0$.

In the present paper we consider such an integral mean of holomorphic on the upper half plane function under the conditions $2\le p <+\infty$ and
\[
\sup_{y>0} M(y)<+\infty
\]
i.e. we consider function $f$ that belongs to the Hardy space $H^p$ of holomorphic on the upper half plane functions. It is well known (see  ``Bounded analytic functions'' by J.~Garnett) that such a function meets the growth condition
$|f(x+iy)|=O(y^{-1/p})$ (both, as $y\to 0^{+}$ and $y\to +\infty$), the integral mean $M$ is a non-increasing function on $(0;+\infty)$, the right hand side limit at $y=0$ is
\[
M(0)=M(0^{+}) =  \lim_{y\to 0^{+}} M(y) = \sup_{y>0} M(y)<+\infty.
\]
where $M(0)$ is defined to be the $L^p$ norm of the boundary values of $f$.

Now, note if $f$ is not the zero function then that $M'<0$ on $(0;+\infty)$. Indeed, if there is a $y_0\in(0;+\infty)$
such that $M'(y_0)=0$ then 
\begin{itemize}
\item on the one hand $M'\ge 0$ on $(y_0;+\infty)$ because of $M''\ge 0$, 
\item on the other hand $M'\le 0$ on $(0;+\infty)$ as $M$ is a non-increasing function.
\end{itemize}
Hence, $M'=0$ on $(y_0;+\infty)$. Therefore, $M''=0$ on $(y_0;+\infty)$. So,
$|f'|^2=u'^2_y+ v'^2_y=0$ and $f=0$ because it is the only constant function that belongs to the Hardy space $H^p$, $2\le p <+\infty$. 

Thus, we have proved the following lemma

\begin{lemma}\label{lemma:integralMean} Let $p$ be such that $2\le p <+\infty$,  $f\in H^p$ ($H^p$ is the Hardy space of holomorphic functions on the upper half plane). If $f$ is not the zero function then the integral mean
\[
M(y)=\int_{-\infty}^{+\infty} |f(x+iy)|^p dx
\]
is a bounded continuous function on $[0;+\infty)$ such that $M\in \smooth{2}{0;+\infty}$ and
\[
M>0,\quad M'<0,\quad M''>0,\quad  (\log M)''\ge 0 \mbox{ on } (0;+\infty).
\]
\end{lemma}

From this point of our paper through the its end $p$ is such that $2\le p <+\infty$ and $H^p$ is the Hardy space of holomorphic functions on the upper half plane,
$f\neq 0$, i.e. $f$ is not the zero function and the integral mean 
\[
M(y)=\int_{-\infty}^{+\infty} |f(x+iy)|^p dx
\]
is defined for $y\in[0;+\infty)$.

\begin{theorem}\label{th:twelve} Let $2\le p <+\infty$,  $f\in H^p\setminus\{0\}$. Assume the functions $\varphi:(0;+\infty)\to (-\infty; +\infty)$ and $h:(0;+\infty)\to (-\infty; +\infty)$  meet the conditions 
\begin{enumerate}\renewcommand{\labelenumi}{({\itshape \roman{enumi}}\,)~}
\item $\varphi \in \smooth{3}{0;+\infty}$, $\varphi(y)=\int_1^y \varphi'(t) dt$
for all $y\in (0;+\infty)$
\item $\varphi'>0$ and 
 $\varphi'^2\varphi''+\varphi \varphi' \varphi''' - 2 \varphi \varphi''^2 \le 0$ on $(0;+\infty)$, 
\item $h(y)=\int_1^y \varphi'(t) M(t) dt$ for all $y\in(0;+\infty)$
\end{enumerate}
Then $\frac{h}{\varphi}$ and  $\log \frac{h}{\varphi}$ both belong to $\smooth{2}{0;+\infty}$ and moreover
\[
\bigl(\frac{h}{\varphi} \bigr)'<0, \bigl(\log \frac{h}{\varphi} \bigr)''>0 \mbox{ on } (0;+\infty)
\]
\end{theorem}

This theorem is a simple corollary of Lemma~\ref{lemma:integralMean} and Theorem~\ref{th1} and we omit the details. Theorem~\ref{th:twelve} holds with each one of the following functions
\begin{itemize}
\item $\varphi(y)=\int_1^y t^{-a} dt$, $a>0$.
\item $\varphi(y)=\int_1^y e^{-t} dt$.
\end{itemize} 

\begin{theorem}\label{th:thirteen} Let $2\le p <+\infty$,  $f\in H^p\setminus\{0\}$. Assume the functions $\varphi:(0;+\infty)\to (-\infty; +\infty)$ and $h:(0;+\infty)\to (-\infty; +\infty)$  meet the conditions 
\begin{itemize}
\item[({\itshape {i}})] $\varphi(y)=\int_0^y \varphi'(t) dt$, $\varphi'>0$
for all $y\in (0;+\infty)$
\item[({\itshape {ii}})] $\lim\limits_{y\to 0^{+}}\frac{\varphi''(y) \varphi(y)}{\varphi'^2(y)}<1$,
 $\varphi'^2\varphi''+\varphi \varphi' \varphi''' - 2 \varphi \varphi''^2 \le 0$ on $(0;+\infty)$, \\ \ \\
 or as an alternative
\item[({\itshape {ii}'}\/)]  $\lim\limits_{y\to 0^{+}}\frac{\varphi''(y) \varphi(y)}{\varphi'^2(y)}=1$,
 $\varphi'^2\varphi''+\varphi \varphi' \varphi''' - 2 \varphi \varphi''^2 < 0$ on $(0;+\infty)$, 
\item[({\itshape {iii}})] $h(y)=\int_0^y \varphi'(t) M(t) dt$ for all $y\in(0;+\infty)$
\end{itemize}
Then $\frac{h}{\varphi}$ and  $\log \frac{h}{\varphi}$ both belong to $\smooth{2}{0;+\infty}$ and moreover
\[
\bigl(\frac{h}{\varphi} \bigr)'<0, \bigl(\log \frac{h}{\varphi} \bigr)''>0 \mbox{ on } (0;+\infty)
\]
\end{theorem}

\begin{proof}  
By Lemma~\ref{lemma:integralMean} $M>0$ and $M'<0$, $(\log M)''\ge 0$ on $(0;+\infty)$.

As in the proof of Theorem~\ref{th2}, by the assumptions it is clear that $\frac{h}{\varphi}$ and  $\log \frac{h}{\varphi}$ both belong to $\smooth{2}{0;+\infty}$. Furthermore,
\[
h(x)-\varphi(x) M(x) = \int_0^x (-1) \varphi(t) M'(t) dt>0 \mbox{ for all } x\in(0;+\infty).
\]
Hence, 
\[
(\tfrac{h}{\varphi})'=\tfrac{-\varphi'}{\varphi^2} (h- \varphi M) <0 \mbox{ on }
(0;+\infty).
\]

From this point of the proof of the theorem to its end let the function $q(x)=1$ for all $x\in(0;+\infty)$ and $A$, $B$, $C$ be as in Definition~\ref{defABC}. 
 
In particular, 
\begin{gather*}
A= \varphi'' \varphi -  \varphi'^2,\quad B=( \varphi' M)' \varphi^2, \quad C=\varphi^2 \varphi'^2 M^2.
\end{gather*} 
and by~\eqref{E1identity} 
\[
E_1=-\varphi^2 \varphi'^3 \bigl( \frac{\varphi'' \varphi}{ \varphi'^2} \bigr)'=-\varphi^2(\varphi'^2\varphi''+\varphi \varphi' \varphi''' - 2 \varphi \varphi''^2 ).
\]

Thus, accordingly to the assumptions of the theorem, $E_1\ge 0$ on $(0;+\infty)$. 
Furthermore,
\begin{itemize}
\item in the case the assumption ({\itshape {ii}}) 
\[
\bigl(\tfrac{A}{\varphi'^2}\bigr)'=\tfrac{\varphi \varphi' \varphi''' + \varphi'^2 \varphi'' - 2 \varphi \varphi''^2}{\varphi'^3}\le 0 \mbox{ on } (0;+\infty)
\]
So, $\tfrac{A}{\varphi'^2}=\frac{\varphi'' \varphi}{\varphi'^2}-1$ decreases on $(0;+\infty)$ and 
\[
\tfrac{A}{\varphi'^2}\le \lim_{x\to 0^{+}} \tfrac{A}{\varphi'^2}<0
\] 
\item in the case the assumption ({\itshape {ii'}\/}) 
\[
\bigl(\tfrac{A}{\varphi'^2}\bigr)'=\tfrac{\varphi \varphi' \varphi''' + \varphi'^2 \varphi'' - 2 \varphi \varphi''^2}{\varphi'^3}< 0 \mbox{ on } (0;+\infty)
\]
So, $\tfrac{A}{\varphi'^2}=\frac{\varphi'' \varphi}{\varphi'^2}-1$ decreases on $(0;+\infty)$ and 
\[
\tfrac{A}{\varphi'^2}< \lim_{x\to 0^{+}} \tfrac{A}{\varphi'^2}=0
\]
\end{itemize}
So, in particular, in both cases $A<0$ on $(0;+\infty)$.

The inequalities $A<0$, $C>0$ on $(0;+\infty)$ imply that the functions $F_{1,2}$
are well defined real valued on $(0;+\infty)$ and together with the inequality 
\[
\left. (A F^2-B F+ C)\right|_{F=\varphi M}\, = - \varphi^3 \varphi' M M' >0
\]
it follows that
\begin{equation}\label{eq:roots}
F_2<0< \varphi M< F_1 \mbox{ on } (0;+\infty).
\end{equation}

We apply the identity~\eqref{derivativehF1} from Lemma~\ref{derivativehF} to obtain
\[
\begin{split}
&(h-F_1)' A(F_2-F_1) q \varphi^2 \varphi' M\\
&\hphantom{(h-F_1} = F_1 \Bigl((F_1-\varphi M)(A^2 (F_1-\varphi M) + E_1 M)+C \varphi^2 \bigl( q\frac{M'}{M}\bigr)' \Bigr)>0
\end{split}
\]
So, 
\begin{equation}\label{eq:eps0}
(h-F_1)'>0 \mbox{ on } (0;+\infty).
\end{equation}

Let $\varepsilon>0$ and
\[
M_{\varepsilon}(y)=\int_{-\infty}^{+\infty} |f(x+(y+\varepsilon)i)|^p dx, \quad \forall y\in(0;+\infty).
\]
Thus, $M_{\varepsilon}(y)=M(y+\varepsilon)$, $M_{\varepsilon}(0)=M(\varepsilon)$,  $M'_{\varepsilon}(0)=M'(\varepsilon)$, $h_{\varepsilon}(y)=\int_0^y \varphi(t) M(t+\varepsilon) dt$ and
\[
B_{\varepsilon}(y)=(\varphi(y)M(y+\varepsilon)+\varphi'(y)M'(y+\varepsilon))\varphi^2(y)
\]
$C_{\varepsilon}=\varphi^2 \varphi'^2 M^2(y+\varepsilon)$,where $y>0$. 
As, it is above, functions 
\[
F_{1,\varepsilon}=\frac{B_{\varepsilon}-\sqrt{B^2_{\varepsilon}-4AC_{\varepsilon}}}{2A},\quad F_{2,\varepsilon}=\frac{B_{\varepsilon}+\sqrt{B^2_{\varepsilon}-4AC_{\varepsilon}}}{2A}
\]
are well defined real valued on $(0;+\infty)$ and together with the inequality 
\[
\left. (A F^2-B_{\varepsilon} F+ C_{\varepsilon})\right|_{F=\varphi M_{\varepsilon}}\, = - \varphi^3 \varphi' M_{\varepsilon} M'_{\varepsilon} >0
\]
it follows that
\[
F_{2,\varepsilon}<0< \varphi M_{\varepsilon}< F_{1,\varepsilon} \mbox{ on } (0;+\infty).
\]

We apply the identity~\eqref{derivativehF1} from Lemma~\ref{derivativehF} to obtain
\[
\begin{split}
&(h_{\varepsilon}-F_{1,\varepsilon})' A(F_{2,\varepsilon}-F_{1,\varepsilon}) q \varphi^2 \varphi' M_{\varepsilon}\\
&= F_{1,\varepsilon} \Bigl((F_{1,\varepsilon}-\varphi M_{\varepsilon})(A^2 (F_{1,\varepsilon}-\varphi M_{\varepsilon}) + E_1 M_{\varepsilon})+C \varphi^2 \bigl( q\frac{M'_{\varepsilon}}{M_{\varepsilon}}\bigr)' \Bigr)>0
\end{split}
\]
So, $(h_{\varepsilon}-F_{1,\varepsilon})'>0$ on $(0;+\infty)$.

Moreover, $\lim_{x\to 0^{+}} (h_{\varepsilon}-F_{1,\varepsilon}) = 0$ because of $\lim_{x\to0^{+}} h_{\varepsilon}=0 $ and 
\begin{align*}
&\lim_{x\to 0^{+}} F_{1,\varepsilon}(x) = \lim_{x\to 0^{+}} \frac{B_{\varepsilon}-\sqrt{B^2_{\varepsilon}-4AC_{\varepsilon}}}{2A} \\
&=\lim_{x\to 0^{+}} \varphi \,\tfrac{\bigl( (\frac{A}{\varphi'^2}+1)M_{\varepsilon} + \frac{\varphi}{\varphi'} M'_{\varepsilon}\bigr) - \sqrt{ \bigl( (\frac{A}{\varphi'^2}+1)M_{\varepsilon} + \frac{\varphi}{\varphi'} M'_{\varepsilon}\bigr)^2 - 4 \frac{A}{\varphi'^2} M^2_\varepsilon } }{ 2 \frac{A}{\varphi'^2} }=0
\end{align*}

Hence, $h_{\varepsilon}-F_{1,\varepsilon}>0$ on $(0;+\infty)$.

Fix an $y>0$. Hence,
\[
h(y)-F_1(y)=\lim_{\varepsilon\to0^{+}}(h_{\varepsilon}-F_{1,\varepsilon})\ge 0. 
\]

So, $h-F_1\ge 0$ on $(0;+\infty)$ and by~\eqref{eq:eps0} it follows that 
\[
h-F_1> 0 \mbox{ on } (0;+\infty).
\]
Therefore, by $A<0$ and~\eqref{eq:roots} it follows that $Ah^2-Bh+C<0$ on 
$(0;+\infty)$.

The second derivative 
\[
(\log \tfrac{h}{\varphi} )''=(q (\log \tfrac{h}{\varphi} )')'=
(\tfrac{q h' \varphi}{h \varphi'})'=(\tfrac{q \varphi M}{h })'=
\tfrac{(-1)(A h^2 - B h + C)}{\varphi^2 h^2} >0
\]
on $(0;+\infty)$.
\end{proof}

Theorem~\ref{th:thirteen} holds with each one of the following functions
\begin{itemize}
\item $\varphi(y)=\int_0^y t^{-a} dt$, $a<1$.
\item $\varphi(y)=\int_0^y e^{-t} dt$.
\end{itemize} 

\begin{theorem}\label{th:fourteen}
Let $2\le p <+\infty$,  $f\in H^p\setminus\{0\}$. Assume the functions $\varphi:(0;+\infty)\to (-\infty; +\infty)$ and $h:(0;+\infty)\to (-\infty; +\infty)$  meet the conditions 
\begin{enumerate}\renewcommand{\labelenumi}{({\itshape \roman{enumi}}\,)~}
\item $\varphi \in \smooth{3}{0;+\infty}$, $\varphi(y)=-\int_y^{+\infty} \varphi'(t) dt$
for all $y\in (0;+\infty)$
\item $\varphi'>0$ and 
 $\varphi'^2\varphi''+\varphi \varphi' \varphi''' - 2 \varphi \varphi''^2 \le 0$ on $(0;+\infty)$, 
\item $h(y)=-\int_y^{+\infty} \varphi'(t) M(t) dt$ for all $y\in(0;+\infty)$
\end{enumerate}
Then $\frac{h}{\varphi}$ and  $\log \frac{h}{\varphi}$ both belong to $\smooth{2}{0;+\infty}$ and moreover
\[
\bigl(\frac{h}{\varphi} \bigr)'<0, \bigl(\log \frac{h}{\varphi} \bigr)''>0 \mbox{ on } (0;+\infty)
\]
\end{theorem}

This theorem is a simple corollary of Lemma~\ref{lemma:integralMean} and Theorem~\ref{th3} and we omit the details. Theorem~\ref{th:fourteen} holds with each one of the following functions
\begin{itemize}
\item $\varphi(y)=-\int_y^{+\infty} t^{-a} dt$, $a>1$.
\item $\varphi(y)=-\int_y^{+\infty} t^{-a}e^{-t} dt$, $a<0$.
\end{itemize}

\begin{theorem}\label{th:fifteen}
Let $2\le p <+\infty$,  $f\in H^p\setminus\{0\}$, 
$\varphi(y)=-\int_y^{+\infty} e^{-t} dt$, $h(x)=-\int_y^{+\infty} e^{-t} M(t) dt$
for all $y\in (0;+\infty)$.
Then $\frac{h}{\varphi}$ and  $\log \frac{h}{\varphi}$ both belong to $\smooth{2}{0;+\infty}$ and moreover
\[
\bigl(\frac{h}{\varphi} \bigr)'<0, \bigl(\log \frac{h}{\varphi} \bigr)''>0 \mbox{ on } (0;+\infty)
\]
\end{theorem}

\begin{proof} 
By Lemma~\ref{lemma:integralMean}, $M\in \smooth{2}{0;+\infty}$, 
\[
M>0,\quad M'<0,\quad M''> 0, \quad (\log M)''\ge 0 \mbox{ on } (0;+\infty).
\]

By the assumptions it is clear that $\frac{h}{\varphi}$ and  $\log \frac{h}{\varphi}$ both belong to $\smooth{2}{0;+\infty}$. 

{\itshape {Claim.}\/} $h-\varphi M>0$ on $(0;+\infty)$.
Indeed, by the assumptions of the theorem 
\begin{itemize}
\item $(h-\varphi M)'=- \varphi M' <0$ on $(0;+\infty)$
\item $M$ decreases on $(0;+\infty)$ and $M>0$. So, the limit \\ $\lim_{y\to{+\infty}} M(x)$ exists and it is a non-negative number.
\end{itemize}
Therefore,  $h-\varphi M$ decreases on $(0;+\infty)$ and 
\[
\lim_{y\to +\infty} (h(y)-\varphi(y) M(y)) = 0.
\]
Hence, $h-\varphi M>0$ on $(0;+\infty)$.

It follows from this Claim that the derivative
\[
(\tfrac{h}{\varphi})'=\tfrac{-\varphi'}{\varphi^2} (h- \varphi M) <0 \mbox{ on }
(0;+\infty).
\]
 
From this point of the proof of the theorem to its end let the function $q(y)=1$ for all $y\in(0;+\infty)$ and $A$, $B$, $C$ be as in Definition~\ref{defABC}. 
 
In particular, 
\begin{gather*}
A= \varphi'' \varphi -  \varphi'^2=0,\quad B=(\varphi'' M + \varphi' M') \varphi^2<0, \quad C=\varphi^2 \varphi'^2 M^2.
\end{gather*} 

By Lemma~\ref{derivativehBC}
\[
\bigl(h-\frac{C}{B}\bigr)'\, B^2 q \varphi' M = C^2 \bigl( q\frac{M'}{M}\bigr)'\ge 0 \mbox{ on } (0;+\infty).
\]
So, $(h-\tfrac{C}{B})'\ge 0$ on $(0;+\infty)$
Therefore, $h-\tfrac{C}{B}$ increases on $(0;+\infty)$ and
\[
\begin{split}
&\lim_{y\to+\infty}(h(y)-\tfrac{C(y)}{B(y)})=
\lim_{y\to+\infty}h(y)-
\lim_{y\to+\infty}\tfrac{C(y)}{B(y)}\\
&=\lim_{y\to+\infty} \tfrac{\varphi^2(y)\varphi'^2(y) M^2(y)}{(\varphi''(y)M(y)+\varphi'(y)M'(y))\varphi^2(y)}=
\lim_{y\to+\infty} \tfrac{\varphi'(y) M(y)}{\frac{\varphi''(y)}{\varphi'(y)}+\frac{M'(y)}{M(y)}}=0
\end{split}
\]
because in the last equation the numerator tends to $0$ and the denominator
tends to sum of $(-1)$ and a non-positive number.

{\itshape {Claim.}\/} $h-\tfrac{C}{B}<0$ on $(0;+\infty)$.
 
Indeed, if there is $y_0\in(0;+\infty)$ such that  $h(y_0)-\tfrac{C(y_0)}{B(y_0)}=0$ then $h-\tfrac{C}{B}=0$ on $(y_0;+\infty)$. $(\tfrac{M'}{M})'=0$
on $(y_0;+\infty)$ which means that $\tfrac{M'}{M}$ is a non zero constant on $(y_0;+\infty)$ (because of $f$ is not the zero function). The equation $h'=(\tfrac{C}{B})'$ then gives us $\tfrac{M'}{M}=1$ on $(y_0;+\infty)$.
Therefore, $M(y)=e^y.{\mathrm {const}}$ and by $M> 0$, $M'\le 0$ it follows that $M=0$  on $(y_0;+\infty)$ which is impossible because of $f$ is not the zero function.

Now, the second derivative
\[
(\log \tfrac{h}{\varphi})''=\tfrac{B(h-\frac{C}{B})}{h^2 \varphi^2}>0 \mbox{ on }
(0;+\infty). 
\]
\end{proof}

Note that in some specific cases it seems reasonable to change parts of the proofs with an argument for $-Bh+C<0$. In particular, such an approach will make us to use part of the proof of Lemma 7. However, we prefer not to do this. 

\begin{example}[An auxiliary example]
If 
\[
\varphi(y)=-\int_y^{+\infty} e^{t-e^{t}} dt, M(y)=e^{y^2}, h(y)=-\int_y^{+\infty} \varphi'(t) M(t) dt
\]
then $(\tfrac{h}{\varphi})'>0$ and $(\log \tfrac{h}{\varphi})'' >0$ on  $(0;+\infty)$.
\end{example}


\begin{thebibliography}{1}

\bibitem{HardyInghamPolya1927}
G.~H. {Hardy}, A.~E. {Ingham}, and G.~{P\'olya}.
\newblock {Theorems concerning mean values of analytic functions.}
\newblock {\em {Proc. R. Soc. Lond., Ser. A}}, 113:542--569, 1927.

\bibitem{WangXiao1405}
Ch. Wang and J.~Xiao.
\newblock Gaussian integral means of entire functions: logarithmic convexity
  and concavity.
\newblock \url{http://arxiv.org/abs/1405.6193v1} (2014).

\bibitem{WangXiao1301}
Ch. Wang and J.~Xiao.
\newblock Gaussian integral means of entire functions.
\newblock {\em Complex Anal. Oper. Theory}, 8:1487--1505, 2014.
\newblock \url{http://arxiv.org/abs/1301.0349v3} (2013).

\bibitem{WangXiao1600}
Ch. Wang and J.~Xiao.
\newblock Addendum to `gaussian integral means of entire functions'{}.
\newblock {\em Complex Anal. Oper. Theory}, 10:495--503, 2016.

\bibitem{WangXiaoZhu1308}
Ch. Wang, J.~Xiao, and K.~Zhu.
\newblock Logarithmic convexity of integral means for analytic functions ii.
\newblock {\em J. Aust. Math. Soc.}, 98:117--128, 2015.
\newblock \url{http://arxiv.org/abs/1308.4881v1} (2013).

\bibitem{WangZhu1101}
Ch. Wang and K.~Zhu.
\newblock Logarithmic convexity of integral means for analytic functions.
\newblock {\em Math. Scand.}, 114:149--160, 2014.
\newblock \url{http://arxiv.org/abs/1101.2998v1} (2011), Zbl:1294.30104.

\bibitem{XiaoXu2011}
J.~Xiao and W.~Xu.
\newblock Weighted integral means of mixed areas and lengths under holomorphic
  mappings.
\newblock {\em Anal. Theory Appl.}, 30:1--19, 2014.
\newblock \url{http://arxiv.org/abs/1105.6042v1} (2011).

\bibitem{XiaoZhu2011}
J.~Xiao and K.~Zhu.
\newblock Volume integral means of holomorphic functions.
\newblock {\em Proc. Amer. Math. Soc.}, 139:1455--1465, 2011.
\newblock MathSciNet: 2748439.

\end{thebibliography}
\end{document}